\documentclass[10pt,centertags,reqno]{amsart}       
\usepackage{amssymb}   
\usepackage{graphicx}

\newtheorem{thm}{Theorem}

\newtheorem{cor}{Corollary}
\newtheorem{prop}{Proposition}

\theoremstyle{definition}

\theoremstyle{remark}
\newtheorem*{rem*}{Remark}
\newtheorem{rem}{Remark}

\newcommand{\mn}{{\mathbb N}}

\newcommand{\mc}{{\mathbb C}}

\renewcommand{\rho}{\varrho}

\renewcommand{\Re}{\operatorname{Re}}

\newcommand{\hil}{\mathcal{H}}

\begin{document}

\title[]{Some remarks on upper bounds for Weierstrass primary factors and their application in spectral theory} 

\begin{abstract}
We study upper bounds on Weierstrass primary factors and discuss their application in spectral theory. One of the main aims of this note is to draw attention to works of Blumenthal and Denjoy from 1910, but we also provide some new results and some numerical computations of our own.
\end{abstract}

\author[M. Hansmann]{Marcel Hansmann}
\address{Faculty of Mathematics\\ 
Chemnitz University of Technology\\
Chemnitz\\
Germany.}
\email{marcel.hansmann@mathematik.tu-chemnitz.de}

\maketitle

\section{Introduction}

This short note is concerned with the Weierstrass primary factors 
\begin{equation}
  \label{eq:3}
E_0(z) :=(1-z), \qquad  E_n(z):=  (1-z)\exp\left( \sum_{k=1}^{n} \frac{z^k}{k}\right),  \quad n \in \mn,   
\end{equation}
defined for $z \in \mc$. These factors  play an important role in complex analysis, most notably in Weierstrass' construction of entire functions with prescribed zeros \cite[p.77-124]{MR0218192}. For instance, Weierstrass showed that given a sequence $(z_l)_{l \in \mn} \subset \mc \setminus\{0\}$ with the property that $\sum_{l} |z_l|^{-n-1} < \infty$ for some $n \in \mn_0$, the canonical product 
 $$ P(z) = \prod_l E_{n}\left( \frac z {z_l} \right), \quad z \in \mc,$$
 defines an entire function which vanishes exactly at the points $z_l$. For many applications (see e.g. \cite{MR1172909} and references therein) it is important to control the growth of such canonical products and this requires suitable bounds on the primary factors $E_n$. For this reason, in the present note we will study the quantities
\begin{equation}
  \label{eq:12}
  C_{n,\alpha} := \sup_{z \in \mc \setminus \{0,1\}} \frac{\ln|E_{n}(z)|}{|z|^{n+\alpha}}, 
\end{equation} 
for $n \in \mn$ and $\alpha \in [0,1]$ (respectively $n = 0$ and $\alpha \in (0,1]$). That this supremum is finite will be discussed below. In other words, $C_{n,\alpha}$ is the smallest number such that 
\begin{equation}
  \label{eq:8}
|E_{n}(z)| \leq \exp(C_{n,\alpha} |z|^{n+\alpha}), \qquad z \in \mc.
\end{equation}

The existence of bounds of the form (\ref{eq:8}) had already been established in works of Lindel\"of \cite{zbMATH02659983} in 1902 and a comprehensive study of the constants $C_{n,\alpha}$ had been carried out by Blumenthal \cite{zbMATH02633798} and Denjoy \cite{zbMATH02633799} in 1910. However, it seems that only few mathematicians are/were aware of these studies so that in the following years parts of Blumenthal and Denjoy's results have been re-proven over and over again (see Remark \ref{rem:2} below). One of the main aims of the present note is to make the results of these two authors more  widely known. In addition, we will also provide a few results of our own. 
 
Our interest in the constants $C_{n,\alpha}$ originated from an application in spectral theory, namely, the study of regularized determinants of linear operators (see, e.g., \cite{b_Gohberg69, b_Simon05, MR917067, MR1744872, MR2201310, Hansmann15}). Let us briefly indicate what this is about: If $K$ is a compact linear operator on a Hilbert space $\hil$, whose singular numbers $s_n(K)$ are in $l_p(\mn)$ for some $p>0$, then the $p$-regularized determinant of $I-K$, where $I$ denotes the identity operator on $\hil$, is defined as
\begin{equation*}
  {\det}_p(I-K) := \prod_{l} E_{\lceil p \rceil -1} (\lambda_l(K)). 
\end{equation*}
Here $\lceil p \rceil := \min \{ n \in \mn : n \geq p\}$  and $|\lambda_1(K)| \geq |\lambda_2(K)| \geq \ldots$ denote the discrete eigenvalues of $K$, counted according to algebraic multiplicity. In particular, setting
\begin{equation}
  \label{eq:9} 
\Gamma_p := C_{\lceil p \rceil -1, p+1-\lceil p \rceil}  
\end{equation}
we can use (\ref{eq:8}) (and Weyl's inequality \cite{MR0030693})  to estimate
\begin{equation*}
  |{\det}_p(I-K)| \leq \exp \left( \Gamma_p \sum_n |\lambda_n(K)|^p \right) \leq \exp \left( \Gamma_p \sum_n s_n(K)^p \right).
\end{equation*}

Such regularized determinants play an important role in the spectral analysis of compact and compactly perturbed linear operators, since they allow to transfer the problem of analyzing the spectrum of a linear operator to the problem of studying the zero-set of a holomorphic function  (the function $\lambda \mapsto {\det}_p(I-\lambda K)$ is entire), which in turn is intimately connected to its growth behavior. It is for this reason that the constants $\Gamma_p$ appear in a large variety of eigenvalue estimates for linear operators, so a precise knowledge of their values is crucial. To mention just one example, in \cite{MR3296588} it was shown that for a bounded operator $B=A+K$ on a complex Banach space $X$, $K$ being a compact perturbation, one has the following upper bound on the number of discrete eigenvalues of $B$ in $\{ \lambda : |\lambda| > s\}$: For every $p>0$
\begin{equation}
  \label{eq:2}
  N_B(s) \leq \Gamma_p \cdot R_p \cdot \frac{s}{(s-\|A\|)^{p+1}} \sum_{n} a_n^p(K), \qquad s > \|A\|,
\end{equation}
where $R_p$ is an explicitly known constant and $a_n(K)$ denotes the $n$th approximation number of $K$. For other appearances of $\Gamma_p$ in eigenvalue estimates (sometimes with a different notation), see, e.g. \cite{MR2481997, MR2559715, MR2500508, HK10, MR2879252, MR3077277, MR3157981, Frank_Sabin_15, MR3177329, MR3320831, Hansmann15, frank15, frank15b, frank16, MR3606512}.  

The results from below will allow us to compute the $\Gamma_p$'s numerically (apparently, this has not been done before). We conclude this introduction with a plot of the result.  


 \begin{figure}[h]
 \centering 
 \includegraphics[width=0.58\textwidth]{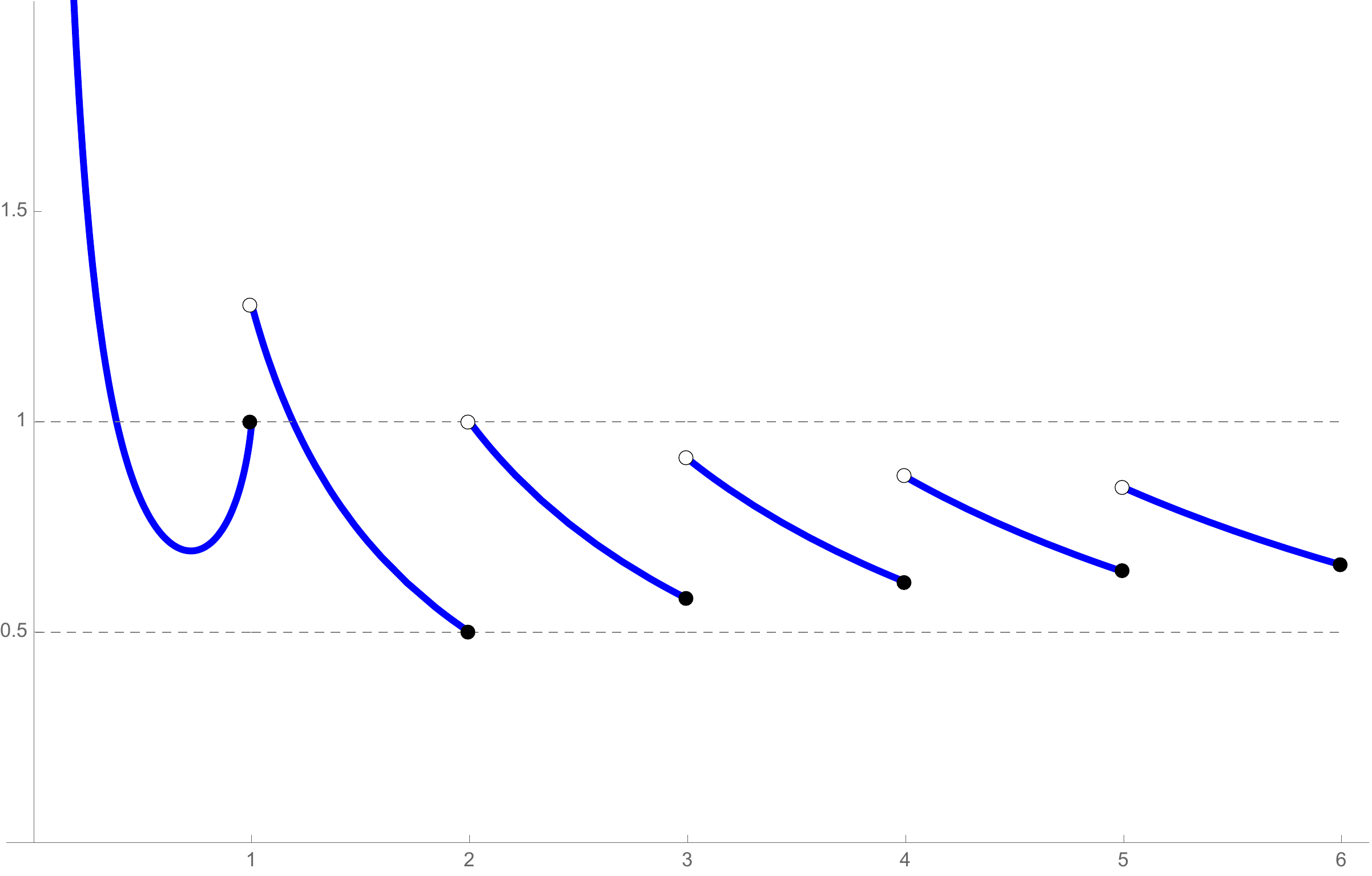}   
 \caption{The graph of $p \mapsto \Gamma_p$, evaluated numerically.}
 \vspace*{-10pt}
 \end{figure}

\section{The case $n \geq 1$}

We start with a look at the case $n \geq 1$. So for $n \in \mn$ and $\alpha \in [0,1]$ let us consider the function
$$ g: z \mapsto \frac{\ln|E_{n}(z)|}{|z|^{n+\alpha}}.$$
Then the following facts can be easily checked:
\begin{enumerate}
    \item $g$ is continuous in $\mc \setminus \{0,1\}$.
 \item $g(z)$ tends to $-\infty$ for $z \to 1$.
 \item Setting $z=re^{i\theta}$ with $r<1$ we can use the Taylor expansion of $\log(1-z)$ to obtain
\begin{eqnarray*}
\ln|E_{n}(z)| = \Re\left(\log(1-z) + \sum_{k=1}^n \frac{z^k}{k} \right) 
= - \sum_{k=n+1}^\infty \frac{r^k \cos(k\theta)}{k},
\end{eqnarray*}
which shows that in case $\alpha \in [0,1)$ the map $g$ can be continuously extended to $z=0$ (setting $g(0)=0$). Moreover, if $\alpha = 1$ then $\limsup_{z \to 0} g(z) = \frac 1 {n+1}$. 
    \item We have $$ \limsup_{|z| \to \infty} g(z) = \left\{
  \begin{array}{cl}
    0, & \alpha \in (0,1], \\
    \frac 1 n, & \alpha = 0,
  \end{array}\right.$$
and if $\alpha = 0$ then $g(r) > 1/n$ for $r >2$.
\end{enumerate} 
The next result is due to Blumenthal and Denjoy.

\begin{prop}[Blumenthal \cite{zbMATH02633798}, Denjoy  \cite{zbMATH02633799}]\label{prop1}
  Let $n \in \mn$ and $\alpha \in [0,1]$. Then the function
$$ (0,\infty) \ni r \mapsto \max_{|z|=r} g(r) = \max_{|z|=r}   \frac{\ln|E_{n}(z)|}{r^{n+\alpha}}$$
is monotone increasing on $(0,1+\frac 1 n]$. Moreover, for $r \geq 1+\frac 1 n$ we have
$$ \max_{|z|=r}  \ln|E_{n}(z)| = \ln |E_n(r)|.$$ 
\end{prop}
For a proof of the proposition we refer to the original works or, alternatively, to a recent re-proof given by Merzlyakov in \cite{zbMATH06594447}, Theorem 1.\\

The previous proposition, together with (1)-(4) from above, shows that $C_{n,\alpha}$ is indeed finite and that $$C_{n,\alpha} = \max_{z \in \mc \setminus \{0,1\}} g(z), \qquad n \in \mn, \quad \alpha \in [0,1].$$ As noticed in the introduction, this observation goes (at least) back to work of Lindel\"of \cite{zbMATH02659983} in 1902 (see also Pringsheim \cite{zbMATH02653595}).
A further consequence of Proposition \ref{prop1} is the following result. 
\begin{cor}\label{cor:1}
   Let $n \in \mn$ and $\alpha \in [0,1]$. Then  
   \begin{equation}
     \label{eq:13}
     C_{n,\alpha}= \max_{r \geq 1 + \frac 1 n} \frac{\ln|E_n(r)|}{r^{n+\alpha}}.
   \end{equation}
Moreover, the mapping $ [0,1] \ni \alpha \mapsto C_{n,\alpha}$
is monotone decreasing and convex.
\end{cor}
The observation that the previous mapping is convex seems to be new.
\begin{proof}
Just note that for $r>1$ the mapping $[0,1] \ni \alpha \mapsto r^{-\alpha}$ is monotone decreasing and convex.   
\end{proof}
The identity (\ref{eq:13}) easily allows us to compute the constants $C_{n,\alpha}$ numerically. The following figure shows the result, using Mathematica's FindMaximum-Routine. Here each rectangle corresponds to the graph of $[0,1] \ni \alpha \mapsto C_{n,\alpha}$ for the specified $n \in \mn$.
 
 \begin{figure}[h]
 \centering
 \includegraphics[width=0.65\textwidth]{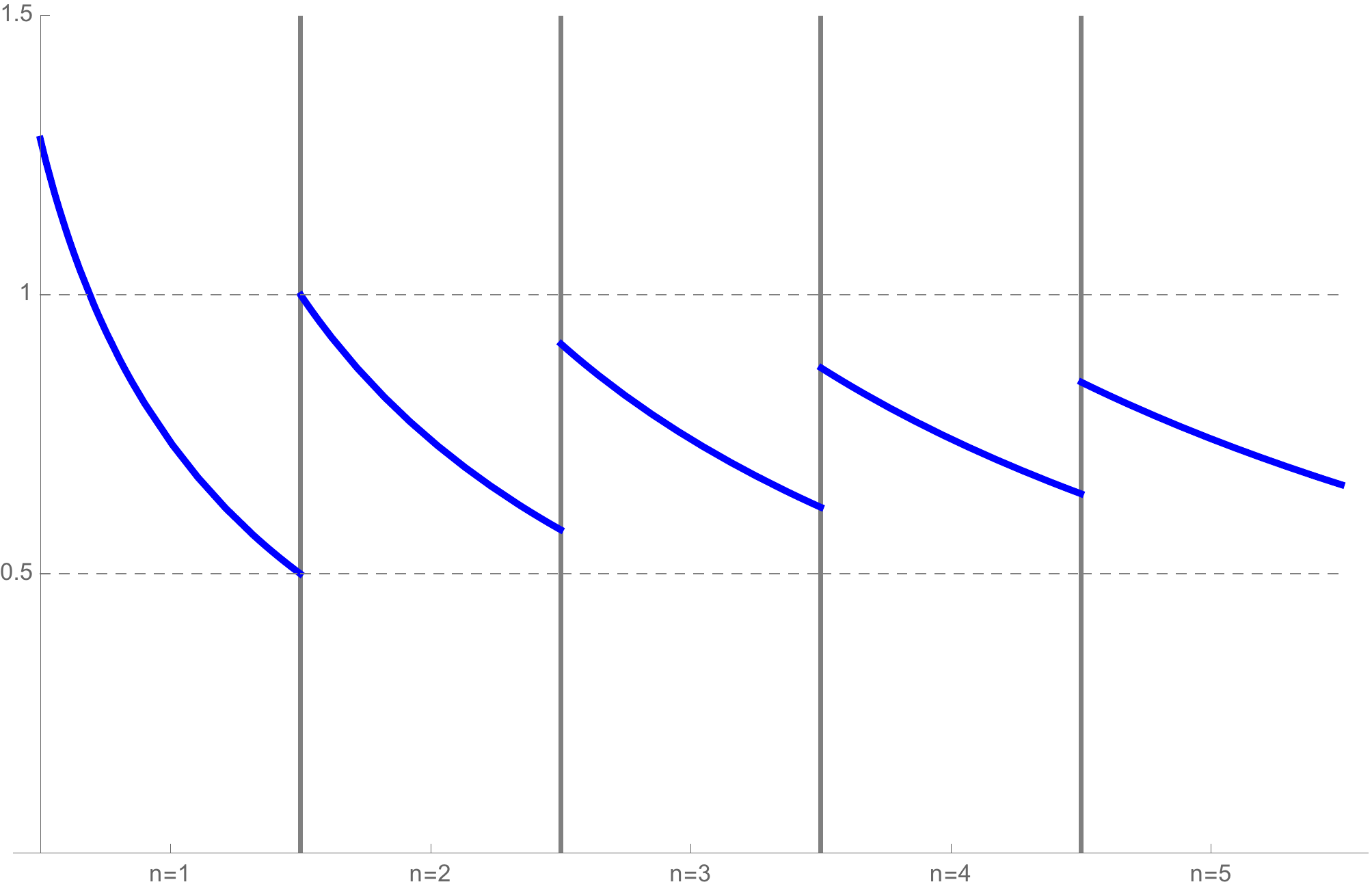}   
 \caption{A numerical evaluation of $[0,1] \ni \alpha \mapsto C_{n,\alpha}$ for $n \geq 1.$}
 \end{figure}   

Some of the features reflected in the previous figure can be proven rigorously. In the following theorem $W: [-e^{-1},\infty) \to [-1,\infty)$ denotes the principal branch of the Lambert function, i.e. the inverse of the strictly increasing function $$[-1,\infty) \ni x \mapsto xe^x \in [-e^{-1},\infty).$$ 

\begin{thm}[Blumenthal \cite{zbMATH02595638}, Denjoy \cite{zbMATH02633799}, Cohn \cite{MR0315094}]\label{thm1}
The following holds: 
  \begin{enumerate}
      \item[(i)] $\mn \ni n \mapsto C_{n,1}$ is monotonically increasing. Moreover, the sequence 
    \begin{equation}
      \label{eq:4}
      g_n:= \frac{n+1} n C_{n,1}, \qquad n \in \mn,
    \end{equation}
is monotonically decreasing and $0 < g_n \leq 1$ for all $n\in \mn$.
    \item[(ii)]  $ \mn \ni n \mapsto C_{n,0}$ is monotonically decreasing.
    \item[(iii)] For every $\alpha \in [0,1]$ we have
$$ \lim_{n \to \infty} C_{n,\alpha} =: x_0^{-1} \approx 0.7423,$$
where $x_0$ is the unique positive zero of the function $e^x/x-\operatorname{Ei}(x)$. Here $\operatorname{Ei}$ denotes the exponential integral.
      \item[(iv)] $C_{1,0}=1+W(e^{-1})\approx 1.27846,$ $C_{1,1}=\frac 1 2,$  and $C_{2,0}=1$.
  \end{enumerate} 
\end{thm}
Here (i) is due to Blumenthal, (ii) is due to Cohn and (iii) is due to Denjoy. The evaluation of the constants in (iv) is straightforward, but can also, for instance, be found in the work of Cohn. It should be noted that Cohn did his work in 1973, probably unaware of the previous work of Blumenthal and Denjoy, and that he also provided a proof of (iii). Most of the above results can also be found in a recent paper by Merzlyakov \cite{zbMATH06594447}.

\begin{rem}\label{rem:2}
The bound $$C_{n,1} \leq n/(n+1),$$ implicitly contained in (i), had previously been obtained both by Blumenthal and Denjoy in their 1910 papers \cite{zbMATH02633798} and \cite{zbMATH02633799}. Unaware of this work, Smithies \cite{MR0004699} (in 1941) and Brascamp \cite{MR0291889} (in 1969) re-proved this inequality for $n=1$ and $n=3$, respectively. The general result, however, seems to have been widely forgotten over time. Standard text-books on entire functions (such as, e.g., Nevanlinna  \cite{zbMATH02526051})  usually contain only the much weaker but easier to prove inequality that $C_{n,1} \leq 3e(2+\ln(n))$. Eventually the general bound was re-proven by Marchetti \cite{MR1244940} in 1993. There we can also find the fact that the constants $g_n$ appearing in (\ref{eq:4}) can be bounded above by
$$ h_n:=\exp\left(- \frac{n-1}{4(n+1)}\left\{1 + \left(1 + 2 \left(1+\operatorname{cosec} \frac \pi{n+1} \right)^{-1} \right)^n\right\}^{-1} \right).$$ 
However, $\lim_{n \to \infty} h_n \approx 0.9995,$ whereas $\lim_{n \to \infty} g_n \approx 0.7423$. In 2013 and 2016, respectively, the general estimate was again re-proven by Gil' \cite{MR3035142} (with the exception of the case $n = 2$) and by Merzlyakov \cite{zbMATH06594447}. 
\end{rem}
 
As a consequence of Theorem \ref{thm1} and Corollary \ref{cor:1} we obtain the following result, which seems to be new. 
\begin{cor}
Let $\alpha \in [0,1]$. Then for every $n \in \mn$
\begin{equation}
  \label{eq:15}
C_{n,1}\leq C_{n,\alpha} \leq (1-\alpha) C_{n,0} + \alpha C_{n,1}.  
\end{equation}
In particular, 
\begin{equation}
  \label{eq:14}
C_{1,\alpha} \leq (1-\alpha)(1+W(e^{-1}))+\frac{\alpha}{2}  
\end{equation}
and for $n \geq 2$
\begin{equation}
  \label{eq:16}
C_{n,\alpha} \leq 1- \alpha \left(1-\min \left(x_0^{-1}, \frac n {n+1} \right) \right) \leq 1,
\end{equation}
with $x_0^{-1} \approx 0.7423$ as in Theorem \ref{thm1}. 
\end{cor}
\begin{proof}
The inequalities (\ref{eq:15}) follow by monotonicity and convexity of $\alpha \mapsto C_{n,\alpha}$. Estimate (\ref{eq:14}) follows by (\ref{eq:15}) and Theorem \ref{thm1}, part (iv). Finally, in the proof of (\ref{eq:16}) we use (\ref{eq:15}) and the facts that, by Theorem \ref{thm1}, $C_{n,0} \leq C_{2,0} = 1$ and $C_{n,1} \leq \min \left( \frac{n}{n+1}, x_0^{-1} \right)$.
\end{proof}
 


\section{The case $n=0$}

In this section we consider the case $n=0$ which can be computed explicitly (and quite easily).  The  results in this section seem to be new.\\

For $\alpha \in (0,1]$ we set 
$$ C_{0,\alpha} := \sup_{z \in \mc \setminus \{0,1\}} \frac{\ln|E_{0}(z)|}{|z|^{\alpha}} = \sup_{ z \in \mc \setminus \{0,1\}} \frac{\ln|1-z|}{|z|^\alpha}.$$
Since $\ln|1-z| \leq \ln(1+|z|)$, with equality for $z \in (-\infty,0)$, we see that
$$ C_{0,\alpha} = \max_{r \geq 0} \frac{\ln(1+r)}{r^\alpha}.$$
In the following theorem $W$ denotes again the principal branch of the Lambert function. 
\begin{thm}\label{thm2} 
The following holds:
\begin{enumerate}
    \item[(i)]    $C_{0,1} = 1$.
    \item[(ii)] If $\alpha \in (0,1)$, then 
\begin{equation}
  \label{eq:5}
C_{0,\alpha} = \frac 1 \alpha  r_\alpha^\alpha (1-r_\alpha)^{1-\alpha},
\end{equation}
where 
 \begin{equation}
   \label{eq:11}
   r_\alpha= -\alpha W\left(-\frac 1 \alpha e^{-\frac 1 \alpha}\right) \in (0,\alpha).
 \end{equation}
\end{enumerate}
\end{thm} 
 \begin{figure*}[h]
 \centering   
 \includegraphics[width=0.6\textwidth]{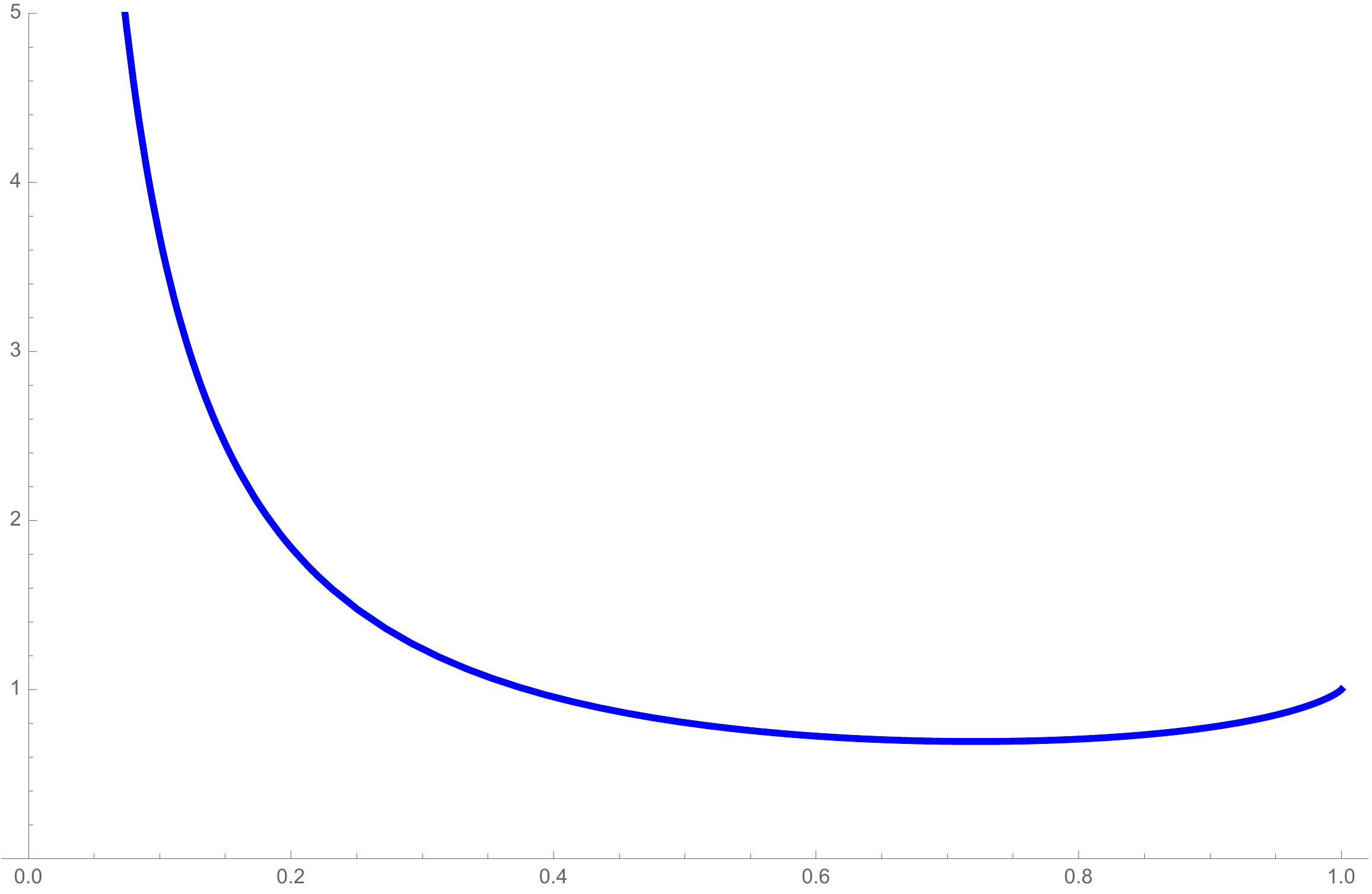}    
 \caption{The graph of $(0,1] \ni \alpha \mapsto C_{0,\alpha}$}
 \end{figure*} 
 \begin{rem}
 Since $W(x) \approx x$ for small $x>0$, we see that 
 $$ \lim_{\alpha \searrow 0} \alpha \cdot C_{0,\alpha} = e^{-1}.$$ 
 Moreover, by maximizing $r^\alpha(1-r)^{1-\alpha}$ over $r \in (0,\alpha)$ we obtain
\begin{equation*}
  C_{0,\alpha} \leq \left( \frac 1 \alpha -1 \right)^{1-\alpha}, \qquad \alpha \in (0,1].
\end{equation*}
 \end{rem}
\begin{proof}[Proof of Theorem \ref{thm2}]
(i) Since $\ln(1+r) \leq r$ for all $r \geq 0$ and 
$$ \lim_{r \searrow 0} \frac{\ln(1+r)}{r}=1$$
we obtain $C_{0,1}=1$. 

(ii) Now let $0<\alpha < 1$ and set 
$$ g(r):= \frac{\ln(1+r)}{r^\alpha}, \quad r > 0.$$
Note that $g$ is positive and $\lim_{r \searrow 0} g(r)= \lim_{r \to \infty} g(r)=0$ so it has a maximum in $(0,\infty)$. A short computation shows that $g'(r)=0$ iff
$$ r = \alpha (1+r) \ln(1+r).$$ 
This equation has exactly one positive solution $R>0$ which, as a short computation shows, is given by
$$ R=e^{W\left(-\frac 1 \alpha e^{-\frac 1 \alpha}\right)+ \frac 1 \alpha}-1.$$
So the maximal value of $g$ is given by
\begin{eqnarray*}
  g(R) &=&\frac{W\left(-\frac 1 \alpha e^{-\frac 1 \alpha}\right)+ \frac 1 \alpha}{\left(e^{W\left(-\frac 1 \alpha e^{-\frac 1 \alpha}\right)+ \frac 1 \alpha}-1\right)^\alpha} = \frac 1 \alpha \frac{1-r_\alpha}{\left(e^{\frac 1 \alpha \left( 1 -r_\alpha \right)}-1\right)^\alpha}.
\end{eqnarray*}
Now in order to obtain (\ref{eq:5})  we expand this fraction by $r_\alpha^\alpha$ and use that 
\begin{eqnarray*}
 r_\alpha e^{\frac 1 \alpha \left( 1 -r_\alpha \right)} = r_\alpha e^{-\frac{r_\alpha}{\alpha}}e^{\frac 1 \alpha} &=& -\alpha W\left(-\frac 1 \alpha e^{-\frac 1 \alpha}\right) e^{ W\left(-\frac 1 \alpha e^{-\frac 1 \alpha}\right)} e^{\frac 1 \alpha} \\
&=& -\alpha \left(-\frac 1 \alpha e^{-\frac 1 \alpha}\right)e^{\frac 1 \alpha} =1.  
\end{eqnarray*}
\end{proof}

\section*{Acknowledgments}  

I would like to thank Leonid Golinskii for a helpful discussion. This work was funded by the Deutsche Forschungsgemeinschaft (DFG, German Research Foundation) - Project number HA 7732/2-1.




  \bibliographystyle{plain}           
  \bibliography{Bibliography}

\end{document}